\documentclass[11pt]{amsart}

\usepackage[utf8]{inputenc} 
\usepackage{amscd}
\usepackage{a4wide}
\usepackage{hyperref}
\usepackage[active]{srcltx}
\usepackage{pdfsync}
\usepackage{enumerate}

\usepackage[parfill]{parskip}

\newcommand{\RR}{\mathbb{R}}
\newcommand{\CC}{\mathbb{C}}
\newcommand{\D}{\mathbb{D}}

\newcommand{\cpn}{\mathbb{C}P^n_p}
\newcommand{\snp}{\mathbb{S}^{2n+1}_{2p}}
\newcommand{\ep}{\varepsilon}
\newcommand{\bD}{\bar{D}}
\newcommand{\tD}{\tilde{D}}
\newcommand{\mth}{\mathbf{\tilde{H}}(t)}
\newcommand{\mh}{\mathbf{H}(t)}

\newcommand{\mtqm}{\mathbf{\tilde{M}}_q^m(t)}
\newcommand{\mqm}{\mathbf{M}_q^m(t)}
\newcommand{\q}{\mathfrak{q}_1}

\newcommand{\qq}{\mathfrak{q}_2}

 \newtheorem{theorem}{Theorem}[section]
 \newtheorem{cor}[theorem]{Corollary}
 \newtheorem{lemma}[theorem]{Lemma}
 
\newtheorem*{conj*}{Conjecture}
\theoremstyle{definition}
 \newtheorem{definition}[theorem]{Definition}
 \newtheorem{example}[theorem]{Example}

\begin{document}

\title{Hopf Real Hypersurfaces in the Indefinite Complex Projective Space}

\author[M. Kimura]{Makoto Kimura${}^\ddagger$}
\address[M. Kimura]{Department of Mathematics, Faculty of Science, Ibaraki University, Mito, Ibaraki 310-8512, JAPAN}
\curraddr{}
\email{makoto.kimura.geometry@vc.ibaraki.ac.jp}
\thanks{${}^\dagger$ Supported by JSPS KAKENHI Grant Number JP16K05119.}

\author[M.~Ortega]{Miguel Ortega${}^\star$}

\address[M.~Ortega]{Instituto de Matemáticas IEMathUGR,
Departamento de Geometr\'ia y Topolog\'ia, Facultad de Ciencias, Universidad de Granada, 18071 Granada (Spain)}
\curraddr{}
\email{miortega@ugr.es}
\thanks{${}^\star$ Partially supported by the Spanish Ministery of Economy and Competitiveness, and European Region Development Fund, project MTM2016-78807-C2-1-P, and by the Junta de Andaluc\'{\i}a grant FQM-324.}

\date{}


\begin{abstract}
We wish to attack the problems that H.~Anciaux and K.~Panagiotidou posed in  \cite{AP}, for non-degenerate real hypersurfaces in indefinite complex projective space. We will slightly change these authors' point of view, obtaining cleaner equations for the almost contact metric structure. To make the theory meaningful, we construct new families of non-degenerate Hopf real hypersurfaces whose shape operator is diagonalisable, and one Hopf example with degenerate metric and non-diagonalisable shape operator.  Next, we obtain a rigidity result. We classify those real hypersurfaces which are $\eta$-umbilical. As a consequence, we characterize some of our new examples as those whose Reeb vector field $\xi$ is Killing. 
\end{abstract}



\maketitle

\section{Introduction}

The study of real hypersurfaces in indefinite complex projective space seems to be initiated by A.~Bejancu and K.~L.~Dugal in \cite{BD}. However, their point of view was not followed by the authors studying real hypersurfaces in Riemannian non-flat complex space forms. We had to wait for H.~Anciaux and K.~Panagiotidou to obtain new results in \cite{AP} concerning non-lightlike hypersurfaces. In these  papers, the authors link two classical branches of Differential Geometry, namely, real hypersurfaces in complex space forms and semi-Riemannian Geometry.  Needless to say, both branches are very well-known and developed. The study of real hypersurfaces can be regarded as a final product, in the sense that it has few  applications, or rather, it has been a good place to use techniques from other theories. On the contrary, semi-Riemannian Geometry is not only interesting by itself, but a very powerful and fruitful tool to solve many problems in Mathematics and Physics. 

A.~Bejancu and K.~L.~Dugal  paid attention to real hypersurfaces  in (flat) complex space forms, by considering the $(\varepsilon)$-Sasakian and $(\varepsilon)$-cosymplectic structures. On the other hand, H.~Anciaux and K.~Panagiotidou  constructed a beautiful theory for non-degenerate real hypersurfaces in both complex and para-complex  indefinite space forms, mainly in the non-flat cases. In their paper \cite{AP}, one can see the influence of all the theory of real hypersurfaces in non-flat complex space forms.  As a result, they showed that some of the classical and most celebrated results in this theory also hold in their setting. We will follow this second line. 

It is worthwhile to say a few words on real hypersurfaces in non-flat complex space forms. Probably, the first important result is due to R.~Takagi, \cite{Takagi}, where  he classified the homogeneous examples in the complex projective space, obtaining the so-called \textit{Takagi's list}. Since then, hundreds of works about real hypersurfaces have appeared, not only when the ambient space is the complex projective space, but also for the complex hyperbolic space, the quaternionic space forms, the Grassmanian of 2-complex planes, and the complex quadric. It is impossible to give a short and fair list of contributions, so we suggest the reader to check the nice survey by T.~E.~Cecil and P.~J.~Ryan,  \cite{CR2}. 

In this paper, we wish to further develop the ideas that Anciaux and Panagiotidou started, which deal with  non-degenerate real hypersurfaces in non-flat indefinite complex and para-complex space forms. We also would like to study the open problems they posed. However, we will just focus on the indefinite complex projective space $\cpn$ of index $1\leq p\leq n-1$, with some subtle, but important differences. In \cite{AP},  when a real hypersurface had a timelike unit normal vector field,  the authors always changed the metric $g$ by $-g$. Our first decision consist of allowing the normal vector to have its own causal character, without changing the metric, because it is a much  more natural way to work. In Section \ref{preliminaries}, on any non-degenerate real hypersurface $M$, we set the almost-contact structure $(g,\xi,\eta)$ where, as usual, $g$ is the metric, $\xi$ is the Reeb vector field, and $\eta$ is the metrically equivalent 1-form. We recover the natural conditions on the almost contact structure, the structure Gauss and Codazzi equations, and the key lemmatta. All of them are very similar, but  slightly different from those obtained in \cite{AP}. The usual techniques for real hypersurfaces in complex space forms cannot be used directly, because we are moving from a Riemannian to a semi-Riemannian setting. Here, we will adapt the techniques of semi-Riemannian Geometry to these real hypersurfaces. 

Next, since an interesting, meaningful theory need examples, we construct new families of real hypersurfaces in Section \ref{examples}. We will call them \textit{real hypersurfaces of type $A_+$, $A_-$, $B_0$, $B_+$, $B_-$ and $C$}, because they  are somehow similar to those in  Takagi's and Montiel's lists (\cite{Mo}, \cite{Takagi}), but with some important differences. Of course, all our examples are Hopf  $(A\xi=\mu\xi$, $A$ the shape operator). We also call Type $C$ a \textit{horosphere}, due to its shape operator satisfies $A\xi=2\xi$ and $AX=X$ for any $X\perp\xi$, as a famous example in the complex hyperbolic space. In addition, we  exhibit an example of a degenerate (lightlike), Hopf real hypersurface. Thus, we positively answer the first open problem in \cite{AP}, since we show the existence of real hypersurfaces such that $A\xi=\mu\xi$ with $\vert\mu\vert<2$ and $\vert\mu\vert=2$.  We recall that J.~Berndt in \cite{Berndt} and the first author in \cite{K} proved a very useful result, namely, that a real hypersurface in a complex space form is Hopf and has constant principal curvatures if, and only if, it is one of the examples in Montiel's list and Takagi's list, respectively. Then, we hope that a similar result holds true in our setting.
\begin{conj*} Let $M$ be a non-degenerate real hypersurface in $\cpn$ whose shape operator is diagonalisable Then, $M$ is Hopf and all its principal curvatures are constant if, and only, if, $M$ is locally congruent to one of the examples $A_{+}$, $A_{-}$, $B_0$, $B_{+}$, $B_{-}$, or $C$.
\end{conj*}
In Section \ref{ridigity}, first we prove that two real hypersurfaces with the same shape operator are linked by a holomorphic isometry of the ambient space. This was obtained by R.~Takagi in \cite{Takagi} for real hypersurfaces in the Riemannian complex projective space. Next, in Theorem \ref{casiumibilical}, we obtain the list of non-degenerate real hypersurfaces in indefinite complex projective space such that $AX=\lambda X+\rho \eta(X)\xi$, where $\lambda$ and $\rho$ are smooth functions. They were studied in complex space forms by S.~Montiel, \cite{Mo}, and by R.~Takagi, \cite{Takagi2}. Corollary \ref{Killing} is the answer to the second open problem  in \cite{AP}, because we classify those non-degenerate real hypersurfaces such that $A\phi=\phi A$. 

Finally, we would like to thank the referee for some useful comments, specially on the examples. 

\section{Preliminaries}\label{preliminaries}

Let $\CC^{n+1}_p$ be the Euclidean complex space endowed with the following hermitian product and pseudo-Riemannian metric of index $2p$, 
$z=(z_	1,\ldots,z_{n+1})$, $w=(w_1,\ldots,w_{n+1})\in\CC^{n+1}$, 
\begin{equation}\label{innerp} 
g_{\CC}(z,w)=-\sum_{j=1}^p z_j\bar{w}_j+\sum_{j=p+1}^{n+1}z_j\bar{w}_j, \quad 
g=\mathrm{Re}(g_{\CC}),
\end{equation}
where $\bar{w}$ is the complex conjugate of $w\in\CC$.  The natural complex structure will be denoted by $J$. As usual, we define the set $\mathbb{S}^1=\{a\in\CC : a\bar{a}=1\} = \{\mathrm{e}^{i\theta} : \theta\in\mathbb{R}\}$. We consider the hyperquadric 
\[ \mathbb{S}^{2n+1}_{2p}=\{z\in\CC ^{n+1}_p : g(z,z)=1\},
\]
which is a semi-Riemannian manifold of index $2p$. We define the action and its corresponding quotient
\[\mathbb{S}^1\times\mathbb{S}^{2n+1}_{2p}\rightarrow \mathbb{S}^{2n+1}_{2p}, \ (a,(z_1,\ldots,z_{n+1}))\mapsto (az_1,\ldots,az_{n+1}), 
\]
\[ \pi:\mathbb{S}^{2n+1}_{2p}\rightarrow \cpn=\mathbb{S}^{2n+1}_{2p}/\sim.
\]
Let $g$ be the metric on $\cpn$ such that $\pi$ becomes a semi-Riemannian submersion. The manifold $\cpn$ is called the \textit{Indefinite Complex Projective Space}. See \cite{BR} for details. We need $1\leq p\leq n-1$ to avoid $\mathbb{C}P^n$ with either a Riemannian or a negative definite metric. Let $\bar\nabla$ be its Levi-Civita connection. Then, $\cpn$ admits a complex structure $J$ induced by $\pi$, with Riemannian tensor
\begin{align}\nonumber 
\bar{R}(X,Y)Z =& g(Y,Z)X-g(X,Z)Y \\ & +g(JY,Z)JX-g(JX,Z)JY+2g(X,JY)JZ, \label{RR}
\end{align}
for any $X,Y,Z\in TM$. Thus, $\cpn$ has constant holomorphic sectional curvature $4$. 

Let $M$ be a connected, non-degenerate, immersed real hypersurface in $\cpn$. If $N$ is a local unit normal vector field such that $\ep=g(N,N)=\pm 1$, we define the \textit{structure} vector field on $M$ as $\xi=-JN$. Clearly, $g(\xi,\xi)=\ep$. Given $X\in TM$, the vector $JX$ might not be tangent to $M$. Then, we decompose it in its tangent and normal parts, namely
\[ JX = \phi X+\ep\eta(X)N,
\]
where $\phi X$ is the tangential part, and $\eta$ is the 1-form on $M$ such that 
\[\eta(X)=g(JX,N)=g(X,\xi).\]
In addition, $\eta(\phi X)=g(\phi X,\xi)=g(JX,-JN)=0$. This implies $\phi^2X=J\phi X=J(JX-\ep\eta(X)N)=-X-\ep\eta(X)JN$, so that
\[ \phi^2 X = -X+\ep\eta(X)\xi.\]
Next, for each $X,Y\in TM$, simple computations show
\begin{gather*}
g(\phi X,\phi Y) =g(JX-\ep\eta(X)N,JY-\ep\eta(Y)N) 
=g(X,Y)-\ep\eta(X)\eta(Y), \\
g(\phi X,Y)+g(X,\phi Y)=0.
\end{gather*}
From this formula, it is very simple to get $\phi\xi=0$. Also, from $\eta(X)=g(X,\xi)$, we obtain $\eta(\xi)=\ep$. Thus, the set $(g,\phi,\eta,\xi)$ is called an almost contact structure on $M$. 

Next, if $\nabla$ is the Levi-Civita connection of $M$,  we have the Gauss and Weingarten formulae:
\begin{equation}\label{GaussWeingarten}
 \bar\nabla_XY=\nabla_XY+\ep g(AX,Y)N, \quad
\bar\nabla_X N  = -AX,\end{equation}
for any $X,Y\in TM$, where $A$ is the shape operator associated with   $N$. Note that 
$\nabla_X\xi = \bar\nabla_X\xi -\ep g(AX,\xi)N = -J\bar\nabla_XN-\ep g(AX,\xi)N = JAX-\ep g(AX,\xi)N = \phi AX$, 
\[ \nabla_X\xi =\phi AX.\] 
The Codazzi equation is
\begin{align*}
(\nabla_XA)Y-(\nabla_YA)X = \eta(X)\phi Y-\eta(Y)\phi X+2g(X,\phi Y)\xi,
\end{align*}
for any $X,Y\in TM$. Let $R$ be the curvature operator of $M$. Then, by using \eqref{RR}, \eqref{GaussWeingarten}  and the structure  Gauss equation \cite[pag. 100]{ONeill}, we obtain 
\begin{align*}
R(X,Y)Z=&g(Y,Z)X-g(X,Z)Y+g(\phi Y,Z)\phi X-g(\phi X,Z)\phi Y\\
& -2g(\phi X,Y)\phi Z +\ep g(AY,Z)AX -\ep g(AX,Z)AY,
\end{align*}
for any $X,Y,Z\in TM$. As usual, the Ricci tensor is the trace of the Riemann tensor. Given a local orthonormal frame $(E_1,\ldots,E_{2n-1})$, with $\ep_i=g(E_i,E_i)=\pm 1$, we can compute 
$\mathrm{Ric}(X,Y) = \sum_i \ep_i g(R(X,E_i)E_i,Y)
=(2n+1)g(X,Y)-3\ep\eta(X)\eta(Y)+ \ep\mathrm{tr}(A)g(AX,Y)-\ep g(A^2X,Y).
$ 
If we call $S$ the associated metrically equivalent Ricci endomorphism, we obtain 
\[ SX = (2n+1)X-3\ep \eta(X)\xi +\ep \mathrm{tr}(A)AX -\ep A^2X,
\]
for any $X\in TM$. 
\begin{definition} Let $M$ be a real hypersurface in $\cpn$. We will say that $M$ is {\em Hopf} when its structure vector field $\xi$ is everywhere principal, i.~e., it is an eigenvector of $A$. 
\end{definition}
Its associated principal curvature can be defined as $\mu = \ep g(A\xi,\xi)$, and we will call it the \textit{Hopf curvature}. Therefore, it holds $A\xi =\mu \xi$. We recall the following basic results. 
\begin{theorem}\label{mu=constant} \cite{AP}
Let $M$ be a non-degenerate Hopf real hypersurface in $\cpn$ with $A\xi=\mu\xi$. Then, $\mu$ is (locally) constant.
\end{theorem}
Next lemma is essentially included in \cite{AP}, but we adapt it to our needs, and provide a shorter proof. 
\begin{lemma}\label{AphiX} Let $M$ be a non-degenerate Hopf real hypersurface in $\cpn$ with $A\xi=\mu\xi$. Assume that $X\in TM$ is a principal vector with associated principal curvature $\lambda$. Then,
\[(2\lambda-\mu)A\phi X  = (\lambda\mu +2\ep) \phi X.
\]
\end{lemma}
\begin{proof}
Firstly, given any $X\in TM$,
\begin{align*}
& g((\nabla_XA)\xi,\xi)
= g(\nabla_XA\xi,\xi) - g(A\nabla_X\xi,\xi) =
g(\nabla_X(\mu\xi),\xi)-g(\nabla_X\xi,\mu\xi) \\
& = X(\mu)\ep =g((\nabla_{\xi}A)X,\xi)=g(\nabla_{\xi}AX,\xi)-\mu g(\nabla_{\xi}X,\xi) \\
&=\xi(g(AX,\xi))-g(AX,\nabla_{\xi}\xi)-\mu\xi(g(X,\xi)) +\mu g(X,\nabla_{\xi}\xi)\\
&=\xi(\mu)g(X,\xi)+\mu\xi(g(X,\xi))-g(AX,\phi A\xi)-\mu \xi(g(X,\xi)) 
+\mu g(X,\phi A\xi) \\ 
&=\xi(\mu)g(X,\xi).
\end{align*}
This shows
\[ \mathrm{grad}(\mu) = \ep \xi(\mu)\xi.
\]
Secondly, since $A$ is self-adjoint, we  have $g((\nabla_XA)Y,\xi)=g(Y,(\nabla_XA)\xi)$,
for any $X,Y\in TM$. Thus,
\begin{align*}
& g((\nabla_XA)Y,\xi) = g((\nabla_XA)\xi,Y)  = X(\mu)g(\xi,Y) +g((\mu I -A)\phi AX,Y) \\
&= \ep \xi(\mu)g(\xi,X)g(\xi,Y) +g((\mu I -A)\phi AX,Y) \\
&=g((\nabla_YA)X,\xi)+2\ep g(X,\phi Y) \\
&=\ep \xi(\mu)g(\xi,X)g(\xi,Y) +g((\mu I -A)\phi AY,X)+2\ep g(X,\phi Y)\\
&=\ep \xi(\mu)g(\xi,X)g(\xi,Y) -g(Y,A\phi(\mu I -A)X)-2\ep g(\phi X,Y),
\end{align*}
And we obtain 
\[ (\mu I-A)\phi AX = A\phi (A-\mu I)X-2\ep \phi X,
\]
for any $X\in TM$. By inserting $AX=\lambda X$, then
\[ \lambda\mu\phi X-\lambda A\phi X = (\lambda-\mu)A\phi X-2\ep\phi X,
\]
or equivalently,
\[ (\lambda\mu +2\ep) \phi X = (2\lambda-\mu)A\phi X.
\]
\end{proof}
\begin{cor}\label{mu=2}
If $\mu=2\lambda$, then $\ep=-1$, $\vert \mu\vert =2$ and $\vert\lambda\vert=1$.
\end{cor}
\begin{proof} If $\mu=2\lambda$, then we have $0=\lambda\mu+2\ep = 2\lambda^2+2\ep$. Since $\ep=\pm 1$, we immediately obtain the result. 
\end{proof}
We put $\hat{\lambda}$ the principal curvature associated with $\phi X$, provided that $X\perp \xi$ is principal with principal curvature $\lambda$. Lemma \ref{AphiX} and Corollary \ref{mu=2} allow to construct the following table:
\begin{center}
\begin{tabular}{|r|l|} \hline
$\ep=+1$, & $\mu=2\cot(2r)$, $r\in(0,\pi/2)$, $\lambda=\cot(r+\theta)$, $\hat{\lambda}=\cot(r-\theta)$. \\ 
& $\mu=2\tan(2r)$, $r\in(-\pi/4,\pi/4)$, $\lambda=\tan(r+\theta)$, \\ 
& \quad   $\hat{\lambda}=-\cot(r-\theta)$. \\ \hline 
$\ep =-1$, & $\mu=2\coth(2r)$, $r>0$, $\lambda=\coth(r+\theta)$, $\hat{\lambda}=\coth(r-\theta)$; \\ 
& \quad or $\lambda=\tanh(r+\theta)$, $\hat{\lambda}=\tanh(r-\theta)$.  \\
& $\mu=2\tanh(2r)$, $r>0$, $\lambda=\coth(r+\theta)$, $\hat{\lambda}=\tanh(r-\theta)$; \\
& \quad  or   $\lambda=\tanh(r+\theta)$, $\hat{\lambda}=\coth(r-\theta)$. \\
& $\mu=2$, $\lambda\neq 1$, $\hat{\lambda}=1$. \\  \hline 
\end{tabular}
\end{center}

Since $ \pi:\mathbb{S}^{2n+1}_{2p}\rightarrow \cpn$ is a semi-Riemannian submersion and a
\textit{principal fiber bundle with structure Lie group $\mathbb{S}^1$}, we can call it the
\textit{Hopf map}. In addition, given a real hypersurface $M^{2n+1}$ in $\cpn$, then we construct its lift $\tilde{M}^{2n}$, i.e., the following commutative diagram:
\[\begin{CD}
\tilde{M}^{2n} @>>> \mathbb{S}^{2n+1}_{2p} \\
@VVV @VVV \\
M^{2n-1} @>>> \cpn
\end{CD}
\]
We call $\tD$, $D$ and $\bD$, respectively, the Levi-Civita connection of $\tilde{M}$,
$\mathbb{S}^{2n+1}_{2p}$ and $\CC^{n+1}_p$. Let  $\chi:\snp \rightarrow \CC^{n+1}_{p}$ be the position vector, which also plays the role of a unit normal space-like vector field. Note that associated Weingarten endomorphism is $A_{\chi}X=-X$, for any $X\in T\snp$. In general, if $X$ is tangent to $\cpn$ at a given point, we denote $\tilde{X}$ or $X^{\sim}$ its horizontal lift to $\mathbb{S}^{2n+1}_{2p}$. Then, $\tilde{N}$ is going to be the horizontal lift of $N$. This implies that the  horizontal lift of $\xi$ is $\tilde{\xi}=-J\tilde{N}$. The vertical part of $\pi$ is spanned by $J\chi$, which is also space-like. The shape operator of $\tilde{M}$ associated with $\tilde{N}$ is going to be $A_{\tilde{N}}$. Our next target is to compute this operator. Given $X\in TM$, we compute, $g\big(D_{\tilde{X}}\tilde{N},J\chi)
=-g(\tilde{N},\bD_{\tilde{X}}J\chi) = -g(\tilde{N},J\bD_{\tilde{X}}\chi) 
= -g(\tilde{N},J\tilde{X})
= -g(\tilde{\xi},\tilde{X}) = -g(\xi,X).$ Then, we have
$D_{\tilde{X}}\tilde{N} = - A_{\tilde{N}}\tilde{X} = \big( \nabla_XN\big)^{\sim}
+g\big(D_{\tilde{X}}\tilde{N},J\chi)J\chi
= - (AX)^{\sim}-g(\xi,X)J\chi.$
Next, we denote by $\mathcal{H}X$ the horizontal
part of any $X\in T\snp$. In this way, $A_{\tilde{N}}(J\chi) = \mathcal{H}A_{\tilde{N}}(J\chi)  +
g(A_{\tilde{N}}(J\chi) ,J\chi)J\chi$. Given $X\in TM$, we can compute  
$g(A_{\tilde{N}}(J\chi),\tilde{X}) =g(J\chi,A_{\tilde{N}}\tilde{X}) = g(\tilde{\xi},\tilde{X})=g(\xi,X)$, which implies
$\mathcal{H}A_{\tilde{N}}(J\chi) = \tilde{\xi}$. Also, $g(A_{\tilde{N}}(J\chi) ,J\chi) =
-g(D_{J\chi}\tilde{N},J\chi)=g(\tilde{N},\bD_{J\chi}J\chi) =
g(\tilde{N},J^2\chi)=0.$ Summing up, for any $X\in TM$, we obtain 
\begin{equation}\label{shapeopchi}
A_{\tilde{N}}\tilde{X}  = (AX)^{\sim}+g(\xi,X)J\chi, \quad
A_{\tilde{N}}J\chi = \tilde{\xi}.
\end{equation}

It is important to point out  that a real hypersurface in $\cpn$ is a semi-Riemannian
submanifold of arbitrary index, and therefore, its shape operator $A$ might not be diagonalisable

\section{Examples}\label{examples}

\begin{example} \label{typeA} \textbf{Type A.} \ 
Consider $t\in\RR$, $t\neq 0,1$, and $0\leq q\leq p\leq m\leq n+2$, $m>q+1$.  We define the following maps  $\q,\qq:\CC^{n+1}_p\rightarrow\CC^{n+1}_p$. Given $z\in \mathbb{C}^{n+1}_p$, the case $q=0$ and $m=n+2$ is not considered, and 
\begin{trivlist}
\item $\bullet$ if $1\leq q$ and $m\leq n+1$, $\q(z)=(z_1,\ldots,z_{q},0,\ldots,0,z_{m},\ldots,z_{n+1})$, \\
$\qq(z)=(0,\ldots,0,z_{q+1},\ldots,z_{m-1},0,\ldots,0)$; 
\item $\bullet$ if $q=0$ and $m\leq n+1$, $\q (z)=(0,\ldots,0,z_m,\ldots,z_{n+1})$,  \\
$\qq(z)=(z_1,\ldots,z_{m-1},0,\ldots,0)$; 
\item $\bullet$ if $1\leq q$ and $m=n+2$,  $\q(z)=(z_1,\ldots,z_q,0,\ldots,0)$,\\ $\qq(z)=(0,\ldots,0,z_{q+1},\ldots,z_{n+1})$. 
\end{trivlist}
With this notation, we define the following hypersurface
\begin{align*} \mtqm = & \left\{ z=(z_1,\ldots,z_n)\in \snp \ : \ g(\q(z),\q(z))=t \right\} \\
= & \left\{z=(z_1,\ldots,z_n)\in \snp \ : \ g(\qq(z),\qq(z))=1-t \right\}.
\end{align*}
We will study the cases when it is not the empty set. 
This hypersurface is $\mathbb{S}^1$-invariant, so it defines a real hypersurface
$\mqm=\pi(\mtqm)\subset \cpn$. The tangent plane at $z\in\mtqm$ is
\begin{align*}
T_z\mtqm &= \left\{ X\in \mathbb{C}^{n+1}_p \ : \ g(X,\chi_z)=0, \ g(X,\q(z))=0 \right\}.
\end{align*}
We see that $g(\q,\qq)=g(J\q,\qq)=0$. If we call $\ep=\mathrm{sign}((1-t)t)=\pm 1$, we
can choose a unit, normal vector field $\tilde{N}$ in $\snp$ at $z\in\mtqm$, 
\[\tilde{N}_z= \alpha \q(z) +\beta \qq(z), \quad
\alpha = \frac{1-t}{\sqrt{\ep t(1-t)}}, \ \beta= \frac{-t}{\sqrt{\ep t(1-t)}}, \ g(\tilde{N},\tilde{N})=\ep.
\]
It is clear that  $N=\pi_*\tilde{N}$, $\tilde\xi=-J\tilde N=-\alpha J\q-\beta J\qq$, $\xi=\pi_*\tilde\xi$. Now, given $X\in T\mqm$ and its horizontal lift $\tilde X=(X_1,\ldots,X_{n+1})$, we have $(AX)^{\sim}=A_{\tilde{N}}\tilde{X} -g(\xi,X)J\chi
=-\bD_{\tilde{X}}\tilde{N}-g(\xi,X)J\chi$, i.~e.,
\begin{equation}\label{Atilde}
(AX)^{\sim}=-\alpha \q(\tilde{X})-\beta\qq(\tilde{X})-g(\xi,X)J\chi.
\end{equation}
With this, given $X\in T\mqm$ such that $X\perp \xi$, we have
\begin{align*}
&g(AX,\xi) =g(A_{\tilde N}\tilde X,\tilde \xi) =
\alpha^2 g(\q(X),J\q)+\beta^2 g(\qq(X),J\qq) =0.
\end{align*}
In other words, the real hypersurface $\mqm$ in $\cpn$ is Hopf, with $A\xi=\mu\xi$. Next, 
$\ep\mu = g(A\xi,\xi) = g(A_{\tilde{N}}\tilde{\xi},\tilde{\xi}) = g(-\alpha \q(\tilde{\xi})-\beta\qq(\tilde{\xi}),-\alpha J\q-\beta J\qq) =\frac{2t-1}{\ep\sqrt{\ep t(1-t)} }.$ In this way,
\[ \mu = \frac{2t-1}{\sqrt{\ep t(1-t)} }.
\]
We call $\tilde{\D}$ the horizontal lift of $\D=\mathrm{Span}\{\xi\}^{\perp}\subset T\mqm$. Clearly,  $\tilde{\D}=\mathrm{Span}\{\chi,J\chi,\tilde{N},\tilde{\xi}\}^{\perp}$.

On the other hand, assume that $X\in T\mtqm$. This is equivalent to $g(X,\chi)=g(X,\tilde{N})=0$, which imply $0=g(\q(X),\q)+g(\qq(X),\qq) = \alpha g(\q(X),\q)+\beta g(\qq(X),\qq)$. Since $\alpha\neq \beta$, we obtain
$g(\q(X),\q)=g(\qq(X),\qq)=0$. If, in addition, we take $X\perp \{J\chi,\tilde{\xi}\}$, we obtain $0=g(\q(X),J\q)=g(\qq(X),J\qq)$. We take now $X\in \D$, and its horizontal lift $\tilde{X}\in \tilde{\D}$. We claim that $\q(X)\in \tilde{\D}$. Indeed, $g(\q(X),\chi)=g(\q(X),\q+\qq)=g(\q(X),\q)=0$. Similarly, we have  $g(\q(X),\tilde{N})=\alpha g(\q(X),\q)+\beta g(\q(X),\qq)=0$, and also we obtain $g(\q(X),J\chi)=g(\q(X),\tilde{\xi})=0$. The same conditions hold for $\qq$. In other words, we can restrict $\q,\qq:\tilde{\D}\rightarrow\tilde{\D}$. Now, due to \eqref{Atilde}, $-\alpha$ and $-\beta$ are the other principal curvatures of $\mqm$. In this way, the eigenspaces are  $V_{-\alpha}=\pi_*(\tilde{\D}\cap\ker \qq)$ and  $V_{-\beta}=\pi_*(\tilde{\D}\cap\ker\q)$. Next, we can also restrict $\mathfrak{q}_i:\mathrm{Span}\{\chi,J\chi,\tilde{N},\tilde{\xi}\}
\rightarrow \mathrm{Span}\{\chi,J\chi,\tilde{N},\tilde{\xi}\}$. This shows $\dim V_{-\alpha}=\dim (\tilde{\D}\cap\ker\qq) =\dim\tilde{\D}+\dim\ker\qq-\dim(\tilde{\D}+\ker\qq)
=2(m-q-2)$. Similarly, $\dim V_{-\beta}=\dim (\tilde{\D}\cap\ker\q) =\dim\tilde{\D}+\dim\ker\q-\dim(\tilde{\D}+\ker\q)
=2(n+q-m+1)$. Note that $\dim V_{-\alpha}+\dim V_{-\beta}=2(n-1)$. 

Now, we make a study of the principal curvatures by paying attention to the possible values of  $t$. We choose suitable $r>0$ at each case, and introduce some names:
\begin{itemize}
\item[$(A_{+})$] $\ep=+1$, $0<t=\cos^2(r)<1$, $\mu=2\cot(2r)$, $\lambda=-\tan(r)$, $\dim V_{\lambda_1}=2(m-q-2)$, $\lambda_2=\cot(r)$, $\dim V_{\lambda_2} =2(n+q-m+1)$.

\item[$(A_{-})$]   $\ep=-1$, $1<t=\cosh^2(r)$, $\mu=2\coth(2r)$, $\lambda_1=-\tanh(r)$, $\dim V_{\lambda_1}=2(m-q-2)$, $\lambda_2=\coth(r)$, $\dim V_{\lambda_2}=2(n+q-m+1)$. 
\end{itemize}
Note that $\dim V_{\lambda_1}=0$ if, and only if, $m=q+2$, if and only if, $\dim V_{\lambda_2}=2n-2$ (recall that $q\leq p\leq m$ with $m>q+1$.) Similarly, $\dim V_{\lambda_2}=0$ if, and only if, $\dim V_{\lambda_1}=2(n-1)$, if, and only if, $m=n+q+1$. Since $m\leq n+2$, then $q\leq 1$. 
\end{example}

\begin{example} \label{typeB} \textbf{Type B. } Given $t>0$, $t\neq 1$, we consider the polynomial  $Q(z)=-\sum_{j=1}^pz_j^2+\sum_{j=p+1}^{n+1}z_j^2=g_{\mathbb{C}}(z,\overline{z})$, and we define the following hypersurface
\[\mathbf{\tilde{M}}_t=\left\{ z=(z_1,\ldots,z_{n+1})\in \snp \ : \
Q(z)\overline{Q(z)}=t\right\}.
\]
Note that $\overline{Q(z)}=Q(\overline{z})$. As before, this set is invariant under the action of $\mathbb{S}^1$, so $\mathbf{M}_t=\pi(\mathbf{\tilde{M}}_t)$ is a real hypersurface in $\cpn$.  
One can see that
\[T_z\mathbf{\tilde{M}}_t = \left\{ X=(X_1,\ldots,X_{n+1})\in \CC^{n+1}_{p} \ : \
 g(X,z)=g(X,Q(z)\overline{z})=0
\right\}.\]
From this, if we set $\ep=\mathrm{sign}(t(1-t))=\pm 1$, we can obtain a unit normal vector field 
\[ \tilde{N}_z = \frac{1}{\sqrt{\ep t(1-t)}}(Q(z)\overline{z}-t z), \quad z\in\mathbf{\tilde{M}}_t.
\]
It also holds $g(\tilde{N},\tilde{N})=\ep$, as expected. Given $X\in T_z\mathbf{\tilde{M}}_t$,  a simple computation shows
\[ A_{\tilde{N}}X = \frac{-1}{\sqrt{\ep t(1-t)}}\left( 2g_{\CC}(z,\overline{X})\overline{z}+Q(z)\overline{X}-tX\right).
\]
Given $a\in \CC$, $g(aX,Y)=g(X,\bar{a}Y)$, for any tangent vectors $X,Y$. Also, $g(X,Y)=g(\overline{X},\overline{Y})$. We put $\alpha=1/\sqrt{\ep t(1-t)}$.  If $X\perp \tilde{N},\tilde{\xi},\chi,J\chi$, then 
\[ g(X,z)=g(X,iz)=g(X,Q(z)\overline{z})=g(X,iQ(z)\overline{z})=0.
\]
We want to show that $\mathbf{M} _t$ is Hopf. To do so, given $X\in T\mathbf{M}_t$, $X\perp\xi$, we put $Y=\tilde{X} = (X_1,\ldots,X_{n+1})$. Since $\tilde{\xi}_z = i\alpha \big( tz-Q(z)\overline{z})$, then, 
$g(A_{\tilde{N}}Y,\tilde{\xi})_z = g\big(-\alpha(2g_{\CC}(z,\overline{Y})\overline{z}+Q(z)\overline{Y}-tY),\tilde{\xi}\big)
= -\alpha g\big( 2g_{\CC}(z,\overline{Y})\overline{z}+Q(z)\overline{Y},\tilde{\xi}\big)$, 
and so
\begin{align*}
&g\big( 2g_{\CC}(z,\overline{Y})\overline{z}+Q(z)\overline{Y},\tilde{\xi}\big) = 
g\big( 2g_{\CC}(z,\overline{Y})\overline{z}+Q(z)\overline{Y},i\alpha(tz-Q(z)\overline{z}\big) \\
&=2\alpha t g\big( g_{\CC}(z,\overline{Y})\overline{z},iz) 
- 2\alpha g( g_{\CC}(z,\overline{Y})\overline{z},iQ(z)\overline{z}) \\
&\quad  +\alpha t g(Q(z)\overline{Y},iz) - \alpha g(Q(z)\overline{Y},iQ(z)\overline{z}).
\end{align*}
To make the computations shorter,  we will use suitable $\beta_l=\pm 1$ in \eqref{innerp}, so that 
$g_{\mathbb{C}}(X,Z)=\sum_l \beta_l z_l\bar{w}_l$. Next,
\begin{align*}
&g\big( g_{\CC}(z,\overline{Y})\overline{z},iz) = 
\mathrm{Re}\Big( \sum_l \beta_l \big(\sum_j \beta_jz_jX_j\big)\overline{z}_l(-i)\overline{z}_l\Big) \\
& =\mathrm{Re}\Big( -i \sum_j\beta_jX_j\overline{ Q(z)\overline{z}_j} \Big)  
=g(-iY,Q(z)\overline{z})=0. \\
&g\Big( g_{\CC}(z,\overline{Y})\overline{z},iQ(z)\overline{z}\Big) 
= \mathrm{Re}\Big(\sum_l\beta_l g_{\CC}(z,\overline{Y})\overline{z}_l\,\overline{ i Q(z) \overline{z}_l} \Big)
\\ 
&  = \mathrm{Re}\Big(\sum_l\beta_l g_{\CC}(z,\overline{Y})\overline{z}_l\,(-i)\overline{Q(z)} z_l \Big)
= -\mathrm{Re}\Big( g_{\CC}(z,\overline{Y}) i \overline{Q(z)}\Big) \\
&=- g(i\,\overline{Q(z)}z,\overline{Y}) =  g(i Q(z)\overline{z},Y)=0. 
\end{align*}
In a similar way, we obtain $g(Q(z)\overline{Y},iz)=g(Q(z)\overline{Y},iQ(z)\overline{z})=0$. All of them imply $g(AX,\xi)=g(A_{\tilde{N}}Y,\tilde{\xi})=0.$

Now, we want to compute the  principal curvature $\mu$ associated with $\xi$. We put $X=\tilde{\xi}=i\alpha \big( tz-Q(z)\overline{z})$, then, $\overline{X}=i\alpha(\overline{Q(z)}z-t\overline{z})$. With this, we have 
\begin{align*} &g_{\CC}(z,\overline{X}) = g_{\CC}\Big(z,i\alpha\big(\overline{Q(z)}z-t\overline{z}\big)\Big) 
=-i\alpha g_{\CC}\Big(z,\overline{Q(z)}z\Big)+ti\alpha g_{\CC}(z,\overline{z}) \\
&=-i\alpha Q(z) g_{\CC}(z,z) +ti\alpha Q(z) = (t-1)i\alpha Q(z).
\end{align*}
Next, 
\begin{align*}
&2g_{\CC}(z,\overline{X})\overline{z}+Q(z)\overline{X} 
= 2(t-1)i\alpha Q(z)\overline{z}+Q(z)i\alpha \Big(\overline{Q(z)}z-t\overline{z}\Big) \\
&
=(t-2)i\alpha Q(z)\overline{z}+i\alpha t z  = \tilde{\xi}+(t-1)i\alpha Q(z)\overline{z}.
\end{align*}
Now, $A_{\tilde{N}}\tilde{\xi} = \alpha (t-1)\tilde{\xi} -\alpha^2 (t-1) i Q(z)\overline{z}.$ 
Next, we compute 
\begin{align*}
&g(iQ(z)\overline{z},\tilde{\xi}) = g(iQ(z)\overline{z},i\alpha(tz-Q(z)\overline{z})) \\
& = \alpha t g(Q(z)\overline{z},z) -\alpha g(Q(z)\overline{z},Q(z)\overline{z}) \\
&=\alpha t \mathrm{Re}\Big(Q(z)\overline{Q(z)}\Big) 
-\alpha  \mathrm{Re}\Big(Q(z)\overline{Q(z)}\Big) = \alpha t^2-\alpha t=\alpha t(t-1).
\end{align*}
We come back, 
$\mu\ep =\alpha (t-1)\ep -\alpha^2(t-1)\alpha t(t-1)
= \frac{2(t-1)\ep}{\sqrt{\ep t(1-t)}}.$
Finally,
\[ \mu = \frac{2(t-1)}{\sqrt{\ep t(1-t)}}.\]

When $0<t<1$, we put $t=\sin^2(2r)$ for some  $r\in(0,\pi/4)$, obtaining $\mu = 2\cot(2r)$. When $t>1$, we put $t=\cosh^2(2r)$ for some  $r>0$, obtaining $\mu = 2\tanh(2r)$. 

Next, if we take $X\perp \xi$, its horizontal $Y=\tilde{X}$ is orthogonal to $\{\chi, J\chi, \tilde{N}, \tilde{\xi}\}$. 
From this, at any $z\in\mathbf{\tilde{M}}_t$, it is easy to see $0=g(Q(z)\overline{z},Y) = g(iQ(z)\overline{z},Y)$. Both expressions imply $0=Q(z)\left[-\sum_{j=1}^p\overline{z}_j\overline{Y}_j+\sum_{j=p+1}^{n+1}\overline{z}_j\overline{Y}_j\right]$. As
 $Q(z)\neq 0$, and by taking complex conjugate, we get  $g_{\CC}(z,\overline{Y})=0$. Therefore, 
\[ A_{\tilde{N}} Y = \alpha (t Y-Q(z)\overline{Y}),\quad  Y\perp \{\chi, J\chi, \tilde{N}, \tilde{\xi}\}. 
\]
We see that for any $\lambda\in\mathbb{S}^1$, given $z\in\mathbf{\tilde{M}}_t$, $Q(\lambda z)=\lambda^2Q(z)$. This shows that for each $x\in \mathbf{M}_t$, there exists $z\in\mathbf{\tilde{M}}_t$ such that $x=\pi(z)$ and $Q(z)=\sqrt{t}$. From now, we work at such $z$. Thus,
\[ A_{\tilde{N}} Y = \alpha  (t Y-\sqrt{t}\,\overline{Y}),\quad  Y\perp \{\chi, J\chi, \tilde{N}, \tilde{\xi}\}. 
\]
Next, if $g_{\mathbb{C}}(Y,\overline{z})=g_{\mathbb{C}}(Y,z)=0$, then $\overline{Y}\in\tilde{\mathbb{D}}_z$.  Now,  if $Y=\overline{Y}$, then $A_{\tilde{N}}Y=\alpha (t-\sqrt{t})Y$, whereas if $\overline{Y}=-Y$, then $A_{\tilde{N}}Y=\alpha (t+\sqrt{t})Y$. 
Thus, these are the two other principal curvatures of $\mathbf{M}_t$, with both multiplicities $n-1$. In the following list, the numbers $m_i$ denote the dimension of the associated eigenspaces. We compute the principal curvatures, and introduce some names.
\begin{itemize}
\item[($B_{+}$)] $\ep=+1$,  $0<t=\sin^2(2r)<1$, $\mu=2\cot(2r)$, $\lambda_1=\cot(r)$, $m_1=n-1$, $\lambda_2=\tan(r)$, $m_2=n-1$, $\phi V_{\lambda_1}=V_{\lambda_2}$.

\item[($B_{0}$)] $\ep=-1$, $\mu=\sqrt{3}$, $\lambda=1/\sqrt{3}$, $\dim V_{\mu}=n$, $\dim V_{\lambda}=n-1$, $\phi V_{\mu}=V_{\lambda}$, $\xi \in V_{\mu}$. 

\item[($B_{-}$)] $\ep=-1$, $1< t=\cosh^2(2r)$, $\mu=2\tanh(2r)$, $\lambda_1=\coth(r)$, $m_1=n-1$, $\lambda_2=\tanh(r)$,  $m_2=n-1$, $\phi V_{\lambda_1}=V_{\lambda_2}$. 
\end{itemize}

For $0<t<1$, we show that this real hypersurface is a tube over a complex quadric. Indeed, we consider the set 
\[  \tilde{\mathbb{Q}}^{n-1} =  \left\{ (z_1,\ldots,z_{n+1})\in\snp : Q(z)=0 \right\}, \quad 
\mathbb{Q}^{n-1} = \pi \big(  \tilde{\mathbb{Q}}^{n-1} \big).
\]
Clearly, $\tilde{\mathbb{Q}}^{n-1}$ is the lift of $\mathbb{Q}^{n-1}$. A simple computation gives
\[ T_z\tilde{\mathbb{Q}}^{n-1} = \left\{ X=(X_1,\ldots,X_{n+1})\in \CC^{n+1}_{p} \ : \
 g(X,z)=g_{\CC}(X,\overline{z})=0 \right\}.\]
Thus, an orthonormal normal frame on $\tilde{\mathbb{Q}}^{n-1}$ in $\mathbb{C}^{n+1}_{p}$ is 
$\chi(z)=z$, $\eta_1(z)=\overline{z}$, $\eta_2(z)=i\overline{z}$, for each $z\in\tilde{\mathbb{Q}}^{n-1}$. All three are spacelike. The following geodesic of $\mathbb{S}_{2p}^{2n+1}$ is normal to $\tilde{\mathbb{Q}}^{n-1}$, which projects to a geodesic of $\mathbb{C}P_p^n$ which is normal to $\mathbb{Q}^{n-1}$. For $s,\theta\in\RR$, starting at $z\in\tilde{\mathbb{Q}}^{n-1}$, 
\[ \gamma_{\theta}(s) = \cos(s)z+\sin(s)\big( \cos(\theta)\eta_1(z)+\sin(\theta)\eta_2(z)\big), \ z\in\tilde{\mathbb{Q}}^{n-1}.\]
Since  $z\in\tilde{\mathbb{Q}}^{n-1}\subset \mathbb{S}_{2p}^{2n+1}$,  the following equations hold:
\begin{gather*}
0=Q(z)=g_{\mathbb{C}}(z,\overline{z}), \ \overline{Q(z)}=g_{\mathbb{C}}(\overline{z},z)=0, \ g_{\mathbb{C}}(z,z)=g_{\mathbb{C}}(\overline{z},\overline{z})=1, \\
g_{\mathbb{C}}(z,iz)=-i g_{\mathbb{C}}(z,z)=-i, \ 
g_{\mathbb{C}}(\overline{z},iz)=-i g_{\mathbb{C}}(\overline{z},z)=0. 
\end{gather*}
With them, and by the fact that $g_{\mathbb{C}}$ is bilinear, it is easy to compute
$Q(\gamma_{\theta}(s))=g_{\mathbb{C}}\Big(\gamma_{\theta}(s),\overline{ \gamma_{\theta}(s)}\Big)=\sin(2s)[\cos(\theta)+i\sin(\theta)].$ 
In particular, $Q(\gamma_{\theta}(s)) \overline{Q(\gamma_{\theta}(s))}=\sin^2(2s)$. This means that the set $\mathbf{M}_t$ is a tube of radius $s\in ]0,\pi/4[$ over $\mathbb{Q}^{n-1}$. 
\end{example}
\smallskip
\begin{example} \normalfont \textbf{A degenerate example.}
When $t=1$, the computations are very similar, but there are some differences. Recall $Q(z)=-\sum_{j=1}^pz_j^2+\sum_{j=p+1}^{n+1}z_j^2=g_{\mathbb{C}}(z,\overline{z})$.  We define the $\mathbb{S}^1$ invariant hypersurface 
\[\mathbf{\tilde{M}}_1=\left\{ z=(z_1,\ldots,z_{n+1})\in \snp \ : \
Q(z)\overline{Q(z)}=1, \ 
\mathrm{rank}_{\mathbb{R}}\{z,iz,\overline{z},i\overline{z}\}=4 
\right\}.
\]
The tangent space is 
\[T_z\mathbf{\tilde{M}}_1 = \left\{ X=(X_1,\ldots,X_{n+1})\in \CC^{n+1}_{p} \ : \
 g(X,z)=g(X,Q(z)\overline{z})=0
\right\}.\]
The hypersurface has no tangent plane at the points such that $\{z,iz,\overline{z},i\overline{z}\}$ are $\mathbb{R}$-linearly dependent. For example, $z=(0,\ldots,0,1)$ is one of them. Given $b\in\mathbb{S}^1$, $b\neq \pm 1$ ($p\leq n-1$), then $(1,0,\ldots,b,1)\in\mathbf{\tilde{M}}_1$. It is clear that $\chi(z)=z$ and $\tilde{\chi}(z)=Q(z)\overline{z}$ provide two  normal vector fields to $\mathbf{\tilde{M}}_1$ in $\mathbb{C}_p^{n+1}$. From them, a lightlike normal vector field in $\snp$ is  
\[ \tilde{N}_z = Q(z)\overline{z}-z, \quad z\in\mathbf{\tilde{M}}_1.
\]
Hence, the hypersurface is degenerate. As expected, $g(\tilde{N}_z,z)=g(\tilde{N},Q(z)\overline{z})=0$, that is to say, $\tilde{N}_z\in T_z \mathbf{\tilde{M}}_1$.  Also $g(\tilde{N},J\chi)=0$.  As usual, we put $\tilde{\xi}=-J\tilde{N}$.  

We will use the following natural definition of the shape operator. Given $X\in T_z\mathbf{\tilde{M}}_1$, 
\begin{equation}\label{AN} A_{\tilde{N}}X = -D_X\tilde{N}=
-2g_{\CC}(X,\overline{z})\overline{z}-Q(z)\overline{X}+X. 
\end{equation} 
The distribution $\tilde{\mathbb{D}}_z=T_z\mathbf{\tilde{M}}_1\cap JT_z\mathbf{\tilde{M}}_1$ is clearly complex. Note that $\tilde{N},\tilde{\xi}\in\tilde{\mathbb{D}}$.  Since $0=g_{\mathbb{C}}(X,W)$ iff $0=g_{\mathbb{C}}\big(X,\overline{Q(z)}W\big)$, we can see that $\tilde{\mathbb{D}}_z=\mathrm{Span}\{z,iz,Q(z)\overline{z},iQ(z)\overline{z}\}^{\perp}=\{X\in \mathbb{C}_p^{n+1} : g_{\mathbb{C}}(X,z)=g_{\mathbb{C}}(X,\overline{z})=0\}$. We easily  compute  $A_{\tilde{N}}\tilde{N}=2\tilde{N}$, $A_{\tilde{N}}\tilde{\xi}=0$ and $A_{\tilde{N}} J\chi = \tilde{\xi}$. 

Let $\mathbf{M}_1=\pi(\mathbf{\tilde{M}}_1)$ be the corresponding real hypersurface.  By the previous example, it is a tube of radius $s=\pi/4$ over a totally complex submanifold. 

The normal vector field $N=\pi_{*}(\tilde{N})$ is lightlike, and therefore $N\in T\mathbf{M}_1$. Thus, the induced metric $g$ is degenerate. If we put $\mathbb{D}=\pi_{*}(\tilde{\mathbb{D}})$, then it is complex, and $N,\xi\in\mathbb{D}$. Since the codimension of $\mathbb{D}$ in $T\mathbf{M}_1$ is one, we can choose many locally defined $V\in T\mathbf{M}_1$ such that $T\mathbf{M}_1=\mathbb{D}\oplus \mathrm{Span}\{V\}$, but they cannot be orthogonal, since $g$ is degenerate.   

We define the shape operator as $AX=-\bar{\nabla}_XN$, for any $X\in TM$. Of course, for any $X\in T\mathbf{M}_1$, $A_{\tilde{N}}\tilde{X} = (AX)^{\sim}+g(A_{\tilde{N}}\tilde{X},J\chi)J\chi$. Bearing in mind that $AN=2N$ and $A\xi=0$, similarly to  Example \ref{typeB}, we can compute the principal curvatures of $A$ restricted  to $\mathbb{D}$, but for $\alpha=1$, obtaining $\lambda_1=0$ and $\lambda_2=2$, with both multiplicites $n-1$.  Then, $\mathbf{M}_1$ is Hopf.

Finally, we show that $A$ is not diagonalizable. Suppose that it is so. The remaining case is that there exists $V\not\in\mathbb{D}$ such that $AV=\beta V$ for some locally defined function $\beta$. In this case, $g(V,\xi)\neq 0$. But now, $\beta g(V,\xi)=g(AV,\xi)=g(V,A\xi)=0$. Then, $\beta=0$. We lift up, choosing the point $z\in\mathbf{\tilde{M}}_1$ such that $Q(z)=1$, and putting $Y=\tilde{V}$, we see 
$A_{\tilde{N}}Y=-2g_{\mathbb{C}}(z,\overline{Y})\overline{z} -\overline{Y}+Y=(AV)^{\sim}+g(A_{\tilde{N}}Y,iz)iz=g(V,\xi)iz$. By conjugating and adding, we obtain 
$-2g_{\mathbb{C}}(z,\overline{Y})\overline{z}-2g_{\mathbb{C}}g(\overline{z},Y)z=g(\xi,V)\big(iz-i\overline{z}\big).$ 
Since $g(V,\xi)\neq 0$,  $\{z,\overline{z},iz,i\overline{z}\}$ are linearly dependent. This is a contradiction. 

Remark that this example does not contradict Lemma \ref{AphiX}, since $\xi$ is lightlike. 
\end{example}

\begin{example}\label{horosphere} \textbf{Type C, the Horosphere.} 
Given $t>0$, we define the hypersurface
\[ \mth =\big \{ z=(z_1,\ldots,z_n)\in \snp :
(z_1-z_{n+1})(\overline{z}_1-\overline{z}_{n+1})=t \big\}.
\]
Clearly, $\mth$ is invariant by
the $\mathbb{S}^1$ action, so we can put $\mh=\pi(\mth)$. Since $\snp$ is orientable and
$\mth$ is a closed subset, it is also orientable, and so it is $\mh$. For each point $z\in \mth$, its tangent space is
\[
T_z\mth = \big\{ X=(X_1,\ldots,X_{n+1})\in \CC^{n+1}_p\ :  \
g(X,\chi_z)=0, \ g(X,\zeta_z)=0\big\},
\]
where $\zeta_z=(z_1-z_{n+1},0,\ldots,0,z_1-z_{n+1})$. Moreover, $T_{\pi(z)}\mh =
\pi_{*}\mathcal{H}T_z\mth$. However, $\zeta$ is lightlike. A simple computation shows that
\[ \tilde{N}= \frac{-1}{t}\zeta-\chi,
\]
is a unit, time-like, horizontal,  normal vector field along $\mth$. Thus, $N=\pi_{*}\tilde{N}$ is a unit, time-like, normal vector field along $\mh$, so that the index of $\mth$ and $\mh$ are $2p-1$.

Note that for any $X\in T\mth$, $X=(X_1,\ldots,X_{n+1})$, then $A_{\tilde{N}}X = -D_{X}\tilde{N} = -\bD_{X}\tilde{N}=\bD_{X}\big(\frac{1}{t}\zeta+\chi\big) = \frac{1}{t}\bD_{X}\zeta+X=\frac{1}{t}(X_1-X_{n+1},0,\ldots,0,X_1-X_{n+1})+X$,
\[ A_{\tilde{N}}X = \frac{1}{t}(X_1-X_{n+1},0,\ldots,0,X_1-X_{n+1})+X.
\]
If we take $X\in T_z\mth$, such that $X\perp \tilde{\xi}$ and $X\perp J\chi$,  then $0=g(X,\tilde{\xi}) = -g(J\tilde{N},X) =
g(\frac{1}{t}J\zeta+J\chi,X)=(1/t)g(J\zeta,X)$. By using the expression of $\zeta_z$ for some $z\in \snp$, then $0=\mathrm{Re}( i (z_1-z_{n+1})(\overline{X}_1-\overline{X}_{n+1}))$. In addition, since $X\perp \tilde{N}$, similarly we obtain $0=\mathrm{Re}((z_1-z_{n+1})(\overline{X}_1-\overline{X}_{n+1}))$. By the fact that $z_1\neq z_{n+1}$ due to the definition of $\mth$, then we have $X_1=X_{n+1}$. This is satisfied for any horizontal lift $\tilde{X}$ of any $X\in T\mh$ such that $X\perp \xi$. In this way,  we have, $A_{\tilde{N}}\tilde{X} = \tilde{X}$. But now, by \eqref{shapeopchi}, given $X\in T\mh$, $X\perp \xi$, \[AX=\pi_{*}(A_{\tilde{N}}\tilde{X}) = X.\]
In particular, $\mh$ is a Hopf real hypersurface. Since $A\xi=\mu\xi$ and the fact that $\xi$ is timelike, we have 
$-\mu=g(A\xi,\xi)=g(A_{\tilde{N}}\tilde{\xi},\tilde{\xi}) 
= -g(\bD_{\tilde{\xi}}\tilde{N},\tilde{\xi})
=g\big( \bD_{\tilde{\xi}}(\frac{1}{t}\zeta+\chi),\tilde{\xi}\big) 
= g(\bD_{\tilde{\xi}}(\zeta/t)+\tilde{\xi},\tilde{\xi}) = -1+g(\bD_{\tilde{\xi}}\zeta,\tilde{\xi})/t.$ But $\tilde{\xi} = -J\tilde{N}=J(\frac{1}{t}\zeta+\chi)$, so that $\bD_{\tilde{\xi}}\zeta = (\xi_1-\xi_{n+1},0,\ldots,0,\xi_1-\xi_{n+1})$. By evaluating at $z\in\mth$, $\tilde{\xi}_z=-J\tilde{N}_z = (i/t)(z_1-z_{n+1},0,\ldots,z_1-z_{n+1})+i(z_1,\ldots,z_{n+1})$, so that $\xi_1=i(1+1/t)z_1-(i/t)z_{n+1})$ and $\xi_{n+1}=(i/t)z_1+i(1-1/t)z_{n+1}$, obtaining $\xi_1-\xi_{n+1}=i(z_1-z_{n+1})$. Therefore, $\mu = 1-\frac{1}{t}\mathrm{Re}(-(\xi_1-\xi_{n+1})\bar{\xi}_1+(\xi_1-\xi_{n+1})\bar{\xi}_{n+1} ) = 1+\frac{1}{t} \vert z_1-z_{n+1}\vert^2=2.$ This means $A\xi = 2\xi$.
\end{example}

\section{Results}\label{ridigity}

We consider an immersion $f:M^{2n}\rightarrow\snp$. By shrinking it if necessary, we can assume that there is a globally defined unit vector field $N$ with constant causal character $\ep=g(N,N)=\pm 1$. If we consider two such immersions $f_1,f_2:M^{2n}\rightarrow \snp$, it makes sense to study if the associated Weingarten operators are related in some way, since both satisfy $A_1(p),A_2(p): T_pM^{2n}\rightarrow T_pM^{2n}$. 
\begin{theorem} \label{rigid} Let $f_i:M^{2n-1}_q\rightarrow \cpn$, $i=1,2$ two isometric immersions of the same connected manifold in $\cpn$, with Weingarten endomorphisms $A_1$ and $A_2$. If for each point $p\in M$, $A_1(p)=A_2(p)$, there exists an isometry $\Phi:\cpn\rightarrow\cpn$ such that $f_2=\Phi\circ f_1$. 
\end{theorem}
\begin{proof} We consider the corresponding lifts $\tilde{f}_i:\tilde{M}\rightarrow \snp$, $i=1,2$, via the Hopf fibration $\pi:\snp\rightarrow \cpn$. By \eqref{shapeopchi}, the Weingarten's endomorphisms of $\tilde{f}_i$ coincide everywhere. Therefore, the equations of Gauss and Codazzi are the same for both $\tilde{f}_1$ and $\tilde{f}_2$. As $\snp$ is a space of constant curvature, by a similar way as in Riemannian Space Forms,  there exist an isometry $\hat{\Phi}$ of $\snp$ such that $\hat{\Phi}\circ \tilde{f}_1=\tilde{f}_2$.  $\hat{\Phi}$ can be chosen to be the restriction of an isometry of $\mathbb{C}^{n+1}_p$, so that it is a holomorphic map. Therefore, it holds $\hat{\Phi}_*(\chi) = \chi$ and $\hat{\Phi}_*(J\chi)=J\chi$. This means that we can project $\hat{\Phi}$ to $\cpn$, and obtain our result. 
\end{proof}
\begin{definition} Let $M$ be a real hypersurface in $\cpn$, $n\geq 2$. We say that $M$ is $\eta$-umbilical if its Weingarten endomorphism is of the form $AX=\lambda X+\rho \eta(X)\xi$ for any $X\in TM$, for some functions $\lambda,\rho\in C^{\infty}(M)$. 
\end{definition}

\begin{theorem} \label{casiumibilical} Let $M$ be a connected, non-degenerate, oriented real hypersurface in $\cpn$, $n\geq 2$, such that it is $\eta$-umbilical.  Then, $M$ is locally congruent to one of the following real hypersurfaces:
\begin{enumerate}
\item A real hypersurface of type $A_{+}$, with $m=q+2$, $q\leq p \leq m=q+2$, 
$\mu=2\cot(2r)$ and $\lambda=\cot(r)$,  $r\in(0,\pi/2)$; 
\item A real hypersurface of type $A_{+}$, with $m=n+q+1$, $0\leq q\leq 1$, $\mu=2\cot(2r)$ and $\lambda=-\tan(r)$,  $r\in(0,\pi/2)$; 
\item A real hypersurface of type $A_{-}$, with $m=q+2$, $q\leq p \leq m=q+2$, $\mu=2\coth(2r)$, $r>0$ and $\lambda=\coth(r)$;
\item A real hypersurface of type $A_{-}$, with $m=q+2$, $q\leq p \leq m=q+2$, $\mu=2\coth(2r)$, $r>0$ and $\lambda=\tanh(r)$;
\item A horosphere. 
\end{enumerate}
\end{theorem}
\begin{proof} We notice that the Weingarten endomorphism is diagonalisable, with only two principal curvatures, namely $\lambda$ and $\mu=\lambda+\rho$, with $A\xi=\mu\xi$. By Theorem \ref{mu=constant}, $\mu$ is locally constant. By changing $N$ by $-N$ if necessary, we can suppose $\mu\geq 0$. 

\noindent \textit{Case $\mu\neq 2\lambda$.} By Lemma \ref{AphiX}, it holds $\lambda = \frac{\lambda\mu +2\varepsilon}{2\lambda -\mu}.$ This implies 
\begin{equation} \label{lamu}
\lambda = \frac{\mu \pm \sqrt{\mu^2+4\varepsilon} }{2}.
\end{equation}
This  shows that $\lambda$ is also locally constant. 

If $\varepsilon =+1$, there exists some $r\in(0,\pi/2)$ such that $\mu=2\cot(2r)$. By \eqref{lamu}, either $\lambda=\cot(r)$ or $\lambda=-\tan(r)$, and whose eigenspace satisfies $\dim V_{\lambda}=2n-2$. But these two cases appear in Example \ref{typeA}, as pointed out at the end of it. By Theorem \ref{rigid}, $M$ is locally congruent to one of these examples. 

If $\ep=-1$, there are three possibilities, namely $\mu\in[0,2)$, or $\mu=2$, or $\mu>2$. If $\mu=2\coth(2r)>2$ for some $r>0$, then by \eqref{lamu}, either  $\lambda=\tanh(r)$ or $\lambda=\coth(r)$. We finish these two cases in a similar way as in $\varepsilon=+1$. Next, if $0\leq \mu=2\tanh(2r)<2$ for some $r\geq 0$, then by \eqref{lamu}, we obtain $\lambda=\tanh(2r)\pm \sqrt{(\tanh(2r))^2-1}$. As $(\tanh(2r))^2<1$, we get to a contradiction. Finally, if $\mu=2$, since $\mu\neq 2\lambda$, then $\lambda\neq 1$. By Lemma \ref{AphiX}, there exists a third  principal curvature $\hat{\lambda}=(\lambda\mu -2)/(2\lambda-\mu))=1$. This is a contradiction. 

\noindent \textit{Case $\mu=2\lambda$.} By Corollary \ref{mu=2}, then $\varepsilon=-1$, $\mu=2$ and $\lambda=1$. Then, we have $AX=\eta(X)\xi+X$ for any $X\in TM$. By Example \ref{horosphere} and Theorem \ref{rigid}, $M$ is locally congruent to a horosphere. 
\end{proof}
\begin{cor} \label{Killing} Let $M$ be a non-degenerate real hypersurface in $\cpn$ such that its Weingarten endomorphism is diagonalisable. The following are equivalent:
\begin{enumerate}
\item $\xi$ is a Killing vector field;
\item $A\phi=\phi A$;
\item $M$ is an open subset of one of the following:
\begin{enumerate}
\item A real hypersurface of type $A_{+}$, with $m=q+2$, $q\leq p \leq m=q+2$, 
$\mu=2\cot(2r)$ and $\lambda=\cot(r)$,  $r\in(0,\pi/2)$; 
\item A real hypersurface of type $A_{+}$, with $m=n+q+1$, $0\leq q\leq 1$, $\mu=2\cot(2r)$ and $\lambda=-\tan(r)$,  $r\in(0,\pi/2)$; 
\item A real hypersurface of type $A_{-}$, with $m=q+2$, $q\leq p \leq m=q+2$, $\mu=2\coth(2r)$, $r>0$ and $\lambda=\coth(r)$;
\item A real hypersurface of type $A_{-}$, with $m=q+2$, $q\leq p \leq m=q+2$, $\mu=2\coth(2r)$, $r>0$ and $\lambda=\tanh(r)$;
\item A horosphere. 
\end{enumerate}
\end{enumerate}
\end{cor}
\begin{proof} If $\mathcal{L}_{\xi}$ is the Lie derivative w.r.t. $\xi$ on $M$, given $X,Y\in TM$, we have
\[ \mathcal{L}_{\xi}g(X,Y)=g(\nabla_X\xi,Y)+g(X,\nabla_Y\xi)=g((\phi A-A\phi)X,Y).
\]
This shows the equivalence of items 1 and 2. Next, we assume that $A\phi=\phi A$. Then, $\phi A\xi=A\phi\xi=0$, Therefore, $A\xi\in \ker (\phi)=\mathrm{Span}\{\xi\}$, which means that  $\xi$ is a principal vector with principal curvature $\mu$. First, if $\vert\mu\vert \neq 2$, since $A$ is diagonalisable, given $X\perp \xi$ such that $AX=\lambda X$, by Lemma \ref{AphiX}, $\phi X$ is also a vector field with associated principal curvature $\hat{\lambda}$. Then, $\hat{\lambda}\phi X=A\phi X=\phi AX=\lambda \phi X$. Therefore, $\lambda=\hat{\lambda}$. Now, the Weingarten endomorphism becomes $AX=\lambda X+(\mu-\lambda)\eta(X) \xi$, for any $X\in TM$. Second, essentially, it remains $\mu=2$. By Lemma \ref{AphiX}, $\lambda=1$ is also a principal curvature. Assume that a point $p\in M$, there exists another principal curvature $\rho(p)\neq 1$, with associated vector $Z\perp \xi_z$. By Lemma \ref{AphiX}, $\rho \phi Z=\phi AZ = A\phi Z = \phi Z$. This is a contradiction. We finish the proof by  Theorem \ref{casiumibilical}. 
\end{proof}

\end{document}